\documentclass [a4,11pt]{amsart}
  \usepackage{latexsym}
  \usepackage[all]{xy}
 \usepackage{amsmath,amssymb,amsfonts}

  \def\<{{\langle}}
  \def\>{{\rangle}}

  \def\note#1{{}}

  \def\note#1{}

\newcommand{\Nilp}{\normalfont\mbox{Nilp}\,}

\def\Spec{{\rm Spec}}
\def\Max{{\rm Max}}

  \def\beq{\begin{equation}}
  \def\eeq{\end{equation}}

\newtheorem{theorem}{Theorem}[section]
\newtheorem{lemma}[theorem]{Lemma}

\newtheorem{corollary}[theorem]{Corollary}

\theoremstyle{definition}

\newtheorem{remark}[theorem]{Remark}
\newtheorem{example}[theorem]{Example}

\theoremstyle{Definition and Notation}

\begin{document}
\bibliographystyle{amsplain}


\title[Amalgamated algebra extensions defined by RVN and SFT ...]{Amalgamated algebra extensions defined by
Von Neumann regular and SFT conditions}

\author{Khalid Louartiti}
\address{Khalid Louartiti\\Department of Mathematics, Faculty of Science and Technology of Fez, Box 2202, University S.M. Ben Abdellah Fez, Morocco.
$$ E-mail\ address:\ lokha2000@hotmail.com$$}

\author{Najib Mahdou}
\address{Najib Mahdou\\Department of Mathematics, Faculty of Science and Technology of Fez, Box 2202, University S.M. Ben Abdellah Fez, Morocco.
 $$E-mail\  address:\ mahdou@hotmail.com$$}

\keywords{ Von Neumann regular ring, SFT ring, amalgamated
algebras along an ideal.}

\subjclass[2000]{13D05, 13D02}

\begin{abstract}
Let $f:A\rightarrow B$ be a ring homomorphism and let $J$ be an
ideal of $B$. In this paper, we characterize $R\bowtie^fJ$ to be
Von Neumann regular ring and SFT ring, respectively.
\end{abstract}

\maketitle


 \section{Introduction}
Throughout this paper all rings are assumed to be commutative with
identity element and the dimension of a ring means its Krull
dimension.

Let $A$ and $B$ be two rings, let $J$ be an ideal of $B$ and let
$f: A\rightarrow B$ be a ring homomorphism. In this setting, we
can consider the following subring of $A\times B$:
\begin{center} $A\bowtie^{f}J: =\{(a,f(a)+j)\mid a\in A,j\in
J\}$\end{center} called \emph{the amalgamation of $A$ with $B$
along $J$ with respect to $f$} (introduced and studied by D'Anna,
 Finacchiaro, and  Fontana in \cite{AFF1, AFF2}). This
construction is a generalization of \emph{the amalgamated
duplication of a ring along an ideal} (introduced and studied by
D'Anna and Fontana in \cite{A, AF1, AF2}) and denoted by $A
\bowtie I$. Moreover, other classical constructions (such as the
$A+XB[X]$, $A+XB[[X]]$, and the $D+M$ constructions) can be
studied as particular cases of the amalgamation (\cite[Examples
2.5 and 2.6]{AFF1}) and other classical constructions, such as the
Nagata's idealization (cf. \cite[page 2]{Nagata}), and the CPI
extensions (in the sense of Boisen and Sheldon \cite{Boisen}) are
strictly related to it (\cite[Example 2.7 and Remark 2.8]{AFF1}).

On the other hand, the amalgamation $A\bowtie^f J$ is related to a
construction proposed by Anderson in \cite{Anderson0} and
motivated by a classical construction due to Dorroh \cite{Dorroh},
concerning the embedding of a ring without identity in a ring with
identity. An ample introduction on the genesis of the notion of
amalgamation is given in \cite[Section 2]{AFF1}. Also, the authors
consider the iteration of the amalgamation process, giving some
geometrical applications of it.

 One of the key tools for studying $A\bowtie^fJ$
is based on the fact that the amalgamation can be studied in the
frame of pullback constructions \cite[Section 4]{AFF1}. This point
of view allows the authors in \cite{AFF1,AFF2} to provide an ample
description of various properties of $A\bowtie^f J$, in connection
with the properties of A, J and f. Namely, in \cite{AFF1}, the
authors  studied the basic properties of this construction (e.g.,
characterizations for $A \bowtie^f J$ to be a Noetherian ring, an
integral domain, a reduced ring) and they  characterized those
distinguished pullbacks that can be expressed as an amalgamation.
Moreover, in \cite{AFF2}, they  pursue the investigation on the
structure of the rings of the form $A\bowtie^f J$, with particular
attention to the prime spectrum, to the chain properties and to
the Krull dimension.\\

Recall that a  ring $R$ is called \emph{Von Neumann regular} if
for each  $a\in R$, there exists $x\in R$ such that $axa=a$. It is
proved in \cite[Theorem 2.1]{chhiti} that, for an ideal $I$ of
$R$, $R\bowtie I$ is Von Neumann regular if and only if $R$ is Von
Neumann regular. In section 2, we extend this result to
amalgamated algebra along an ideal. \\

An ideal $I$ is called an $SFT$-ideal if there exists a naturel
number $k$ and a finitely generated ideal $J \subseteq I$ such
that $a^{k} \in  J$
for each $a \in I$. An $SFT$ ring is a ring in wich every ideal is an $SFT$-ideal.\\
 In \cite{A1}, Arnold studies the Krull dimension of a power series ring $R[[x]]$ over a ring $R$ and showed that the dimension is infinite
 unless $R$ is an $SFT$ ring, which forces us to consider only $SFT$ rings when we study finite-dimensional power series extensions.

 For any ring $A$ with finite Krull dimension, we have:\\
$A$ Noetherian $\Longrightarrow  dim A[[X]] <  \infty
\Longrightarrow A$ $SFT$ ring.

  One important family of $SFT$ rings is that of $SFT$ Pr\"ufer domains, which are also called generalized Dedekind domains. The beautiful
  discovery of Arnold is that, for $D$ a finite-dimensional $SFT$-Pr\"ufer domain, $dim D[[x_{1},...,x_{n}]] = n(dim D) + 1$, and
  so $D[[x_{1},...,x_{n}]]$ is an $SFT$ ring \cite{A2}. In 2007, Kang and Park \cite[Theorem 10]{KP} extend Arnold's result to the
  infinite-dimensional case, thus proving that over an infinite-dimensional $SFT$ Pr\"ufer domain $D$, the power series
  ring $D[[x_{1},...,x_{n}]]$ is an $SFT$ ring. In 2010, Park \cite[Theorem 2.4]{P} shows that, if $R$ is an $m$-dimensional $SFT$ globalized
  pseudo-valuation domain, then $dim R[[x_{1},...,x_{n}]] = mn +
  1$ or $mn + n$.

  $SFT$ rings are similar to Noetherian rings and they have many nice properties. It had been a long-standing open question if the
  power series extension of an $SFT$ ring is also an $SFT$ ring. Coykendall's counterexample to this appears in \cite{coy:SFT}. Remark that
  these rings are coherent. Coykendall prove also that a ring $R$ is $SFT$ if and only if each prime
ideal is $SFT$ (\cite{coy:SFT}). \\

In this work, we characterize $R\bowtie^fJ$ to be a Von Neumann
regular ring and $SFT$ ring, respectively.  Our results generate
new and original examples which enrichy the current literature
with new
families of Von Neumann regular rings and $SFT$ rings. \\

\section{ Von Neuman regular amalgamated algebra along an ideal}

This section characterize the amalgamated algebra along an ideal
$R\bowtie^fJ$ to be a Von Neumann regular ring. The main result
(Theorem~\ref{RVN}) enriches the literature with original examples
of Von Neumann regular rings. \\


\begin{theorem}\label{RVN}
Let $A$ and $B$ be two rings, $J$ an ideal of $B$ and let
$f:A\rightarrow B $ be a ring homomorphism. Then, $A\bowtie^fJ$ is
a Von Neumann regular ring if and only if the following statements
holds:
\begin{enumerate}
    \item $A$ is a Von Neumann regular ring.
    \item $\Nilp(B)\cap J=\{0\}$.
    \item Every prime ideal of $B$ which don't contains $J$ is
    maximal.
\end{enumerate}
\end{theorem}


\begin{proof} For each ideals $P$ and $Q$ of $A$ and $B$
respectively, set $P'^f:=P\bowtie^fJ:=\{(p,f(p)+j)\mid p\in P,
j\in J\}$ and $\overline{Q}^f:=\{(a,f(a)+j)\mid a\in A, j\in J,
f(a)+j\in Q\}$.

Assume that $R\bowtie^fJ$ is a Von Neumann regular ring. Then, it
is reduced. Hence, by \cite[Proposition 5.4]{AFF1}, $A$ is reduced
and $\Nilp(B)\cap J=\{0\}$. Let $P$ be a prime ideal of $A$. Then,
by \cite[Proposition 2.6]{AFF2}, $P'^f$ is a prime ideal of
$R\bowtie^fJ$. Hence, it is maximal since $A\bowtie^fJ$ is Von
Neumann regular. Consequently, by \cite[Proposition 2.6]{AFF2},
$P$ is a maximal ideal of $A$. Hence, $A$ is a Von Neumann regular
ring. Thus, $(1)$ and $(2)$ hold. Let $Q$ be a prime ideal of $B$
not containing $J$. By \cite[Proposition 2.6]{AFF2},
$\overline{Q}^f$ is a prime ideal of $A\bowtie^fJ$, and so
maximal. Then, also by \cite[Proposition 2.6]{AFF2}, $Q$ is a
maximal ideal of $B$. Hence, $(3)$ holds.

Conversely, suppose that $(1)$, $(2)$ and $(3)$ hold. By
\cite[Proposition 5.4]{AFF1}, the statements $(1)$ and $(2)$ imply
that $A\bowtie^fJ$ is reduced. Moreover, from \cite[Proposition
2.6 (3)]{AFF2}, $\Spec(A\bowtie^fJ)=\{P'^f\mid P\in \Spec(A)\}\cup
\{\overline{Q}^f\mid Q\in \Spec(B), I\not\subset J\}$ and
$\Max(A\bowtie^fJ)=\{P'^f\mid P\in \Max(A)\}\cup
\{\overline{Q}^f\mid Q\in \Max(B), I\not\subset J\}$. Since $A$ is
Von Neumann regular, then $\Spec(A)=\Max(A)$. On the other hand,
$(3)$ means  that $\{Q\in \Spec(B), I\not\subset J\}=\{Q\in
\Max(B), I\not\subset J\}$. Hence,
$\Spec(A\bowtie^fJ)=\Max(A\bowtie^fJ)$. Consequently,
$A\bowtie^fJ$ is Von Neumann regular, as desired.
\end{proof}


\begin{remark} If $A$ is Von Neumann regular ring and $I$ is an
ideal of $A$ then $\Nilp(A)\cap I=\{0\}\cap I=\{0\}$ and every
prime ideal (in particular these which doesn't contains $I$) is
maximal. Hence, Theorem \ref{RVN} is clearly a generalization of
\cite[Theorem 2.1.]{chhiti}.
\end{remark}


\begin{corollary}\label{cor1}
Let $A$ and $B$ be two rings, $J$ an ideal of $B$ and let
$f:A\rightarrow B $ be a ring homomorphism. If $A$ and $B$ are
both Von Neumann regular rings then so is $A\bowtie^fJ$.
\end{corollary}

\begin{proof} Follows immediately from Theorem \ref{RVN}
\end{proof}


Recall that a ring $R$ is called Boolean ring if $x^2=x$ for each
$x\in R$. Boolean rings are Von Neumann regular.

\begin{example}
Consider the ring  $B=\prod_{i=1}^{n}K_{i}$ with $K_{i}=\{0;1\}$
and $A$ the  subring of stationary sequences of $B$. Set
$J=\bigoplus _{i=1}^{n}K_{i}$ which is an ideal of $B$, and let
$\iota: A \longrightarrow B$ be the canonical embedding of $A$
into $B$. Then $A\bowtie ^{\iota}J$ is a  Von Neumann regular
ring.
\end{example}

\begin{proof}
Follows from Corollary \ref{cor1} since $B$ and $A$ are both
Boolean rings, and then Von Neumann regular rings.
\end{proof}


It is well known that semisimple rings coincide with Noetherian
Von Neumann rings. Hence, we have the following corollary.\\

\begin{corollary}\label{Cor2}
Let $A$ and $B$ be two rings, $J$ an ideal of $B$ and let
$f:A\rightarrow B $ be a ring homomorphism. Then, $A\bowtie^fJ$ is
a semisimple ring if and only if the following statements hold:
\begin{enumerate}
    \item $A$ is a semisimple  ring.
    \item $\Nilp(B)\cap J=\{0\}$.
    \item Every prime ideal of $B$ which doesn't contains $J$ is
    maximal.
    \item $f(A)+J$ is a Noetherian ring.
\end{enumerate}

In particular, if $A$ and $B$ are both semisimple and the ring
homomorphism $\overline{f} : A \rightarrow B/J$ is finite, then
$A\bowtie^fJ$ is semisimple.
\end{corollary}

\begin{proof} By \cite[Proposition 5.6]{AFF1}, $A\bowtie^fJ$ is Noetherian if and only if $A$
and $f(A)+J$ are Noetherian. Then, the desired equivalence follows
directly from Theorem \ref{RVN}.

The last particular statement follows from \cite[Proposition
5.8]{AFF1} and Corollary \ref{cor1}.
\end{proof}

\bigskip


\section{ SFT amalgamated algebra along an ideal}

The main result of this section characterize the amalgamated
algebra along an ideal $R\bowtie^fJ$ to be an $SFT$ ring. This
result (Theorem~\ref{theorem 2}) enriches the literature
with original examples of $SFT$ rings. \\

\begin{theorem}\label{theorem 2}
Let $A$ and $B$ be two rings, $J$ an ideal of $B$ and let $f: A
\rightarrow B $ be a ring homomorphism. Then, $A\bowtie ^{f}J$ is
an $SFT$ ring if and only if $A$ and $f(A) + J$ are both $SFT$
rings.
\end{theorem}


The proof of the theorem involves the following lemmas of
independent interest. \\


\begin{lemma}\label{homo image} Let $R$ be a ring and $K$ be a proper  ideal of $R$.
If $R$ is an $SFT$ ring then so is $R/K$.
\end{lemma}

\begin{proof} Let $\mathcal{J}$ be an ideal of $R/K$.  There
exists an ideal $J$ of $R$ such that $\mathcal{J}=\overline{J}$.
Since $R$ is an $SFT$ ring there exists a finitely generated ideal
$I$ of $R$ and a positive integer $k$ such that $I\subset J$  and
$x^k\in I$ for each $x\in J$. Thus, $\overline{I}$ is a finitely
generated ideal of $R/K$, $\overline{I}\subset \mathcal{J}$  and
$\overline{x}^k\in \overline{I}$ for each $\overline{x}\in
\mathcal{J}$. Hence, $R/K$ is an $SFT$ ring, as desired.
\end{proof}


\begin{lemma}\label{I+J} Let $R$ be a ring. If $I$ and $J$ are two $SFT$
ideals of $R$ then so is $I+J$.
\end{lemma}

\begin{proof} Assume that $I$ and $J$ are $SFT$ ideals of $R$.
Then, there exists finitely generated ideals $I'$ and $J'$ and two
positive integers $k$ and $k'$ such that $I'\subset I$, $J'\subset
J$, $x^k\in I'$ for each $x\in I$ and $y^{k'}\in J'$ for each
$y\in J$. Clearly, $I'+J'$ is a finitely generated subideal of
$I+J$. Moreover, for each $x\in I$ and $y\in J$, we have
\begin{eqnarray*} (a+b)^{k+k'} &=&  \sum^{i=k+k'}_{i=0}C_{k+k'}^{i}a^{i}b^{k+k'-i} \\
 &=&
[\sum^{i=k}_{i=0}C_{k+k'}^{i}a^{i}b^{k-i}]b^{k'}+[\sum^{i=k+k'}_{i=k+1}C_{k+k'}^{i}a^{i-k}b^{k+k'-i}]a^{k}
 \end{eqnarray*} with $C_{k+k'}^{i}
=\displaystyle\frac{(k+k')!}{i!(k+k'-i)!}$. Hence,
$(a+b)^{k+k'}\in I'+J'$. Consequently, $I+J$ is an $SFT$ ideal of
$R$.
\end{proof}


\begin{lemma}\label{I bowtie J} Let $A$ and $B$ be two rings, $J$ an ideal of $B$,
$f:A\rightarrow B $ be a ring homomorphism and let $I$ be an ideal
of $A$. If  $I\bowtie^fJ$ is an $SFT$ ideal of $A\bowtie^fJ$ then
$I$ is an $SFT$ ideal of $A$ with equivalence if $J$ is an $SFT$
ideal of $f(A)+J$.
\end{lemma}

\begin{proof} For a ring $R$, we denote by $L:=<a_i\mid i=1,..,n>_R$
the finitely generated ideal of $R$ generated by $a_1,a_2,..,a_n$.

Assume that $I\bowtie^f J$ is an $SFT$ ideal of $A\bowtie^f J$.
Then, there exists finitely generated ideal $K:=<(i_l,
f(i_l)+j_l)\mid l=1,..,n>_{A\bowtie^fJ}$ of $A\bowtie^fJ$ and  a
positive integer $k$ such that $K\subset I\bowtie^fJ$ and $x^k\in
K$ for each $x\in I\bowtie^fJ$. Set $I'=<i_l\mid l=1,..,n>_A$. It
is clear that $I'\subset I$ and let $i\in I$. Since $(i,f(i))\in
I\bowtie^fJ$, we get  $(i^k,f(i^k))=(i,f(i))^k\in K$. Thus,
$i^k\in I'$. Hence, $I$ is an $SFT$ ideal of $A$.

Assume that  $J$ is an $SFT$ ideal of $f(A)+J$. Then there exists
a finitely generated ideal $J'=<j_l\mid k=1,..,m>_{f(A)+J}$ of
$f(A)+J$ and a positive integer $k$ such that $j^k\in J'$ for each
$j\in J$. Set $\overline{J'}:=<(0,j_l)\mid
i=1,..,m>_{A\bowtie^fJ}$. On the other hand, $I$ is an $SFT$ ideal
of $A$. Then, there exists a finitely generated ideal $I'=<i_l\mid
l=1,..,n>_A$ of $A$ and a positive integer $k'$ such that
$I'\subset I$ and $i^{k'}\in I'$ for each $i\in I$. Set
$\overline{I'}^f:=<(i_l,f(i_l))\mid l=1,..,n>_{A\bowtie^fJ}$.
Clearly, $K:=\overline{I'}^f+\overline{J'}$ is a finitely
generated ideal of $A\bowtie^fJ$ and $K\subset I\bowtie^fJ$.
Moreover, for each $(i,f(i)+j)\in A\bowtie^fJ$,
$(i,f(i)+j)=(i,f(i))+(0,j)$ and $(i,f(i))^k\in \overline{I'}^f$
since $i^k\in I'$ and $(0,j)^{k'}\in \overline{J'}$ since
$j^{k'}\in J'$. Hence, as in the proof of Lemma \ref{I+J}, we can
prove that $(a,f(a)+j)^{k+k'}\in K$. Consequently, $I\bowtie^fJ$
is an $SFT$ ideal of $A\bowtie^fJ$.
\end{proof}


{\noindent\sl Proof of Theorem \ref{theorem 2}.~~} Assume that
$A\bowtie^fJ$ is an $SFT$ ring. By \cite[Proposition 5.1
(3)]{AFF1}, the rings $A$ and $f(A)+J$ are homomorphic images of
$A\bowtie^fJ$. Then, using Lemma \ref{homo image}, they are $SFT$
rings.

Conversely, for each prime ideals $P$ and $Q$ of $A$ and $B$
respectively, set $P'^f:=P\bowtie^fJ:=\{(p,f(p)+j)\mid p\in P,
j\in J\}$ and $\overline{Q}^f:=\{(a,f(a)+j)\mid a\in A, j\in J,
f(a)+j\in Q\}$. Let $P$ be a prime ideal of $A$. Then, by
\cite[Proposition 2.6]{AFF2}, $P'^f$ is a prime ideal of
$R\bowtie^fJ$. Hence, by Lemma \ref{I bowtie J}, it is an $SFT$
ideal of $A\bowtie^fJ$. Let $\overline{Q}^f$ be a prime ideal of
$A\bowtie^fJ$, then $Q_{0}=\overline{Q\bigcap (f(A)+J)}$ is an
ideal of $(f(A)+J)/J$. Hence, there exists a finitely generated
ideal $Q_{0}'=<\overline{(a_{i},f(a_{i})+j_{i})} \mid
i=1,..,n>_{(f(A)+J)/J}$ of $(f(A)+J)/J$ and a positive integer
$k_{0}$ such that $Q_{0}'\subset Q_{0}$ and $x^{k_{0}}\in Q_{0}'$
for each $x\in Q_{0}$.  Set $L_{0}=<(a_{i},f(a_{i})+j_{i})\mid
i=1,..,n>_{A\bowtie^fJ}$. Then
  $I=f^{-1}(J)\bigcap P_{A}(\overline{Q}^f):=\{a \in A  \mid f(a) \in J \ ; \ \exists j \in J  \mid f(a) + j \in Q \ \}$  is an ideal of $A$,
  and so there exists a finitely generated ideal $I'=<a_{i} \mid i=n+1,..,m>$ of $A$, and a positive integer $k_{1}$ such that $I'\subset I$ and
  $x^{k_{1}}\in I'$ for each $x\in I$. Set $L_{1}=
<(a_{i},f(a_{i})+j_{i})\mid i=n+1,..,m>_{A\bowtie^fJ}$.\\ Or
$Q_{1}=Q\bigcap J$ is an ideal of $f(A)+J$. Since $f(A)+J$ is an
$SFT$ ring, then there exists a finitely generated ideal
$Q_{1}'=<j_{i} \mid i=m+1,..,l>_{f(A)+J}$ of $f(A)+J$ and a
positive integer $k_{2}$ such that $Q_{1}'\subset Q_{1}$  and
$x^{k_{2}} \in Q_{1}'$ for each $x\in Q_{1}$. Set
$L_{2}=<(0,j_{i})\mid i=m+1,..,l>_{A\bowtie^{f}J}$ and
$L=L_{0}+L_{1}+L_{2}$. Let $(a,f(a)+j) \in \overline{Q}^f$, then $
\displaystyle \overline{(f(a)+j)}^{k_0}=\sum_{i=1}^{m}
\overline{(f(a_{i})+j_{i})(f(b_{i})+j'_{i})}$.

Set $\beta =\displaystyle (f(a)+j)^{k_0}  -  \sum_{i=1}^{m} (f(a_{i})+j_{i})(f(b_{i})+j \ '_{i}) \in J.$ Then

$\displaystyle f(a^{k_0}- \sum_{i=1}^{m} a_{i}b_{i}) \in J$. Hence
$\displaystyle \alpha =a^{k_0}- \sum_{i=1}^{m} a_{i}b_{i} \in
f^{-1}(J)$. Therefore,
\begin{eqnarray*} \displaystyle (a,f(a)+j)^{k_0}&=&(\alpha+ \sum_{i=1}^{m} a_{i}b_{i}  ,   \beta + \sum_{i=1}^{m} (f(a_{i})+j_{i})(f(b_{i})+j'_{i}))\\
 &=&\sum_{i=1}^{m}(a_{i},f(a_{i})+j_{i})(b_{i},f(b_{i})+j \ '_{i})+(\alpha , \beta).
\end{eqnarray*}
Since $(a,f(a)+j)^{k_0} \in \overline{Q}^f $, then $C_1=\displaystyle \sum_{i=1}^{m}(a_{i},f(a_{i})+j_{i})(b_{i},f(b_{i})+j \ '_{i}) \in \overline{Q}^f $.
Consequently, $(\alpha, \beta) \in\overline{Q}^f$. Therefore, $(\alpha,\beta)=(\alpha,f(\alpha)+e)$ such that $e \in J$ and $f(\alpha)+e \in Q$.\\
Then $\alpha \in I$ and $\alpha^{k_1} =\displaystyle
\sum_{i=n+1}^{m} a_{i}a_{i} \ '$. Thus,
\begin{eqnarray*} (\alpha, \beta)^{k_1} &=&  (\alpha, f(\alpha) +e)^{k_1} =(\alpha^{k_1}, (f(\alpha) +e)^{k_1})= (\alpha^{k_1},
f(\alpha)^{k_1} +e \ '') \\ &=& (\displaystyle \sum_{i=n+1}^{m}
a_{i}a_{i} \ ',\displaystyle \sum_{i=n+1}^{m} f(a_{i})f(a_{i} \
')+e \ '') \\ &=&
(\displaystyle \sum_{i=n+1}^{m} a_{i}a_{i} \ ', \sum_{i=n+1}^{m} (f(a_{i})+j_{i})f(a_{i} \ ')+e \ ')\\ &=&  \left[\sum_{i=n+1}^{m} (a_{i},f(a_{i})+j_{i})(a_{i} \ ', f(a_{i} \ ')\right]+(0,e \ '). \end{eqnarray*}\\
Since $(\alpha,\beta) \in \overline{Q}^f \ ; \ C_2=\displaystyle \sum_{i=n+1}^{m} (a_{i},f(a_{i})+j_{i})(a_{i} \ ', f(a_{i} \ ')\} \in \overline{Q}^f$.
Then, $(0,e \ ') \in \overline{Q}^f$ and $e \ ' \in Q_{1}$. Therefore, $e' \ ^{k_2}=\displaystyle \sum_{i=m+1}^{f} (f(b_{i})+e_i)j_i$. Hence, $(0,e \ ')^{k_2}=\displaystyle \sum_{i=m+1}^{f} (b_{i},f(b_{i})+e_i)(0,j_i) \in L_2 .$\\
Consequently,\begin{eqnarray*} (a, f(a)+j)^{k_{0}+k_{1}+k_{2}} &=&
\left[(a, f(a)+j)^{k_{0}}\right]^{k_{1}+k_{2}}
\\ &=& \left[\sum_{i=1}^{m} (f(a_{i})+j_{i})(f(b_{i})+j \
'_{i})+(\alpha,\beta)\right]^{k_{1}+k_{2}}
 \\ &=&
\left[ \left[ C_1+(\alpha,\beta) \right]^{k_{1}}\right]^{k_{2}}\\ &=& \left[ \sum_{t=0}^{k_1} \left(%
\begin{array}{c}
  t \\
  k_1 \\
\end{array}%
\right) (C_1)^t(\alpha,\beta)^{k_1-t} \right]^{k_2}\\ &=&   \left[\sum_{t=0}^{k_1-1} \left(%
\begin{array}{c}
  t \\
  k_1 \\
\end{array}%
\right) (C_1)^t(\alpha,\beta)^{k_1-t}+(\alpha,\beta)^{k_1}\right]^{k_2} \\ &=& \left[\sum_{t=0}^{k_1-1} \left(%
\begin{array}{c}
  t \\
  k_1 \\
\end{array}%
\right) (C_1)^t(\alpha,\beta)^{k_1-t}+C_2 +(0,e \ ' )\right]^{k_2}\\ &=& \left[\sum_{v=0}^{k_2-1}\left(%
\begin{array}{c}
  v \\
  k_2 \\
\end{array}%
\right)\left[\sum_{t=0}^{k_1-1} \left(%
\begin{array}{c}
  t \\
  k_1 \\
\end{array}%
\right) (C_1)^t(\alpha,\beta)^{k_1-t}+C_2\right]^{v} (0,e \ ' )^{k_2-v}\right]\\
& &+(0,e \ ')^{k_2}. \end{eqnarray*}\\
But $\left[\sum_{t=0}^{k_1-1} \left(%
\begin{array}{c}
  t \\
  k_1 \\
\end{array}%
\right) (C_1)^t(\alpha,\beta)^{k_1-t}+C_2\right]^{v} (0,e \ '
)^{k_2-v}$ $\in L_0+L_1$, and $(0,e \ ')^{k_2} \ \in L_2$.  Hence
$(a, f(a)+j)^{k_{0}+k_{1}+k{2}} \ \in L_0+L_1+L_2$ and so
$\overline{Q}^f$ is an $SFT$ ideal of $A\bowtie^fJ$. Consequently,
$A\bowtie^fJ$ is an $SFT$ ring.
  \hbox{$\sqcup$}\llap{\hbox{$\sqcap$}}\medskip

\bigskip

The following two Corollaries is an immediate consequence of
Theorem~\ref{theorem 2} and Lemma~\ref{homo image}. \\

\begin{corollary}\label{Cor3}
Let $A$ be a ring, $I$ be an ideal of $A$, $J$ be an ideal of $B
:=A/I$ and let $f:A\rightarrow B (=A/I) $ be the canonical
homomorphism. Then, $A\bowtie^fJ$ is an $SFT$ ring if and only if
so is $A$.
\end{corollary}

\begin{corollary}\label{Cor4}
Let $A$ be a ring and $I$ be an ideal of $A$. Then, $A \bowtie I$
is an $SFT$ ring if and only if so is $A$.
\end{corollary}

Noetherian rings are both $SFT$ and coherent rings. In \cite[Page
344]{Bakkari},
Bakkari gives examples of non-coherent $SFT$-rings. Now, we are able to give new examples of non-coherent $SFT$ rings.\\

 \begin{example}
  Let $A$ be a non-coherent SFT ring, $I$ be an ideal of $A$, $J$ be an ideal of $B
:=A/I$ and let $f:A\rightarrow B (=A/I) $ be the canonical
homomorphism. Then:
\begin{enumerate}
    \item $A\bowtie^fJ$ is an SFT ring.
    \item $A\bowtie^fJ$ is not coherent.
\end{enumerate}
\end{example}

 \begin{proof} 1) By Corollary~\ref{Cor3} since $A$ is an $SFT$ ring. \\
 2) $A\bowtie^fJ$ is not coherent by \cite[Theorem 4.1.5]{G} since $A$ is a module retract of $A\bowtie^fJ$ and $A$ is not
coherent.
\end{proof}


  \begin{example}
  Let $A$ be a non-coherent SFT ring and $I$ be an ideal of $A$. Then:
\begin{enumerate}
    \item $A\bowtie I$ is an SFT ring.
    \item $A\bowtie I$ is not coherent.
\end{enumerate}
\end{example}

 \begin{proof} 1) By Corollary~\ref{Cor4} since $A$ is an $SFT$ ring. \\
 2) $A\bowtie I$ is not coherent by \cite[Theorem 4.1.5]{G} since $A$ is a module retract of $A\bowtie I$ and $A$ is not
coherent.
\end{proof}


  \begin{example}
  Let $A$ be a non-coherent SFT ring, $E$ an $A$-module, $B:=A \propto E$ be the trivial ring extension of $A$ by $E$, $f:A\rightarrow B$ be the canonical
homomorphism ($f(a) =(a,0)$) and set $J :=0 \propto E$. Then:
\begin{enumerate}
    \item $A\bowtie^fJ$ is an SFT ring.
    \item $A\bowtie^fJ$ is not coherent.
\end{enumerate}
\end{example}

\begin{proof} 1) By Theorem~\ref{theorem 2} and \cite[Theoreme 3.1]{Bakkari} since
 $f(A) + J = B (=A \propto E)$. \\
 2) $A\bowtie^fJ$ is not coherent by \cite[Theorem 4.1.5]{G} since $A$ is a module retract of $A\bowtie^fJ$ and $A$ is not
coherent.
\end{proof}

\bigskip

\bibliographystyle{amsplain}

\end{document}